\documentclass[12pt]{article}

\usepackage[english]{babel}
\usepackage[utf8]{inputenc}

\usepackage[margin=1.13 in, top=1.1 in, bottom= 1.2 in]{geometry}
\makeatletter
\g@addto@macro\normalsize{%
  \setlength\abovedisplayskip{7pt}
  \setlength\belowdisplayskip{7pt}
  \setlength\abovedisplayshortskip{7pt}
  \setlength\belowdisplayshortskip{7pt}
}
\makeatother

\usepackage{tocloft}

\usepackage{amsfonts}
\usepackage{mathrsfs}
\usepackage{bbm}
\usepackage{latexsym}
\usepackage{math dots}
\usepackage{amssymb}
\usepackage{mathtools}
\usepackage{relsize}
\usepackage{lipsum}

\usepackage{todonotes}
\usepackage{enumitem}
\setlist{nolistsep} 	
\usepackage{amsthm}

\usepackage{xcolor}
\definecolor{Color1}{rgb}{0.0, 0.42, 0.47}
\definecolor{Color2}{rgb}{0.78, 0.11, 0.0}

\usepackage{titlesec}
\titleformat{\section}
  {\large\center\bfseries}
  {\thesection.}{.7em}{}
\titlespacing*{\section}{0pt}{3.5ex plus 0ex minus 0ex}{1.5ex plus 0ex}
\titleformat{\subsection}
  {\center\bfseries}
  {\thesubsection.}{.7em}{}
\titlespacing*{\subsection}{0pt}{3.5ex plus 0ex minus 0ex}{1.5ex plus 0ex}
\titleformat{\subsubsection}
  {\center\bfseries}
  {\thesubsubsection.}{.7em}{}
\titlespacing*{\subsubsection}{0pt}{3.5ex plus 0ex minus 0ex}{1.5ex plus 0ex}

\addto\captionsenglish{}

\usepackage{titling}
\makeatletter
\renewenvironment{abstract}{
\begin{center}
{\bfseries \large\abstractname\vspace{\z@}}
\end{center}
\quotation
}
\makeatother

\usepackage[linktocpage=true]{hyperref}
\usepackage[capitalize]{cleveref}
\hypersetup{citecolor = Color1,colorlinks,
			linkcolor = black,
			urlcolor = Color2}

\newtheoremstyle{plain}{3mm}{3mm}{\slshape}{}{\bfseries}{.}{.5em}{}
\newtheoremstyle{definition}{2mm}{2mm}{}{}{\bfseries}{.}{.5em}{}
\theoremstyle{plain} 
	
\newtheorem{theorem}{Theorem}[section]

\newtheorem{lemma}[theorem]{Lemma}
\newtheorem{corollary}[theorem]{Corollary}
\theoremstyle{definition} 
\newtheorem{definition}[theorem]{Definition}
\newtheorem{remark}[theorem]{Remark}
\newtheorem{example}[theorem]{Example}

\theoremstyle{plain} 
\newcounter{MainTheoremCounter}

\newtheorem{maintheorem}[MainTheoremCounter]{Theorem}

\theoremstyle{plain} 

\usepackage{chngcntr}

\numberwithin{equation}{section}

\allowdisplaybreaks

\newcommand{\Cesaro}{Ces\`{a}ro }

\newcommand{\N}{\mathbb{N}}
\newcommand{\Z}{\mathbb{Z}}
\newcommand{\R}{\mathbb{R}}
\newcommand{\C}{\mathbb{C}}

\renewcommand{\epsilon}{\varepsilon}
\renewcommand{\leq}{\leqslant}
\renewcommand{\geq}{\geqslant}
\renewcommand{\setminus}{\backslash}

\renewcommand{\subset}{\subseteq}

\usepackage[normalem]{ulem}

\usepackage[makeroom]{cancel}

\newcommand{\E}{\operatorname{\mathbb{E}}}

\newcommand{\oo}{\infty}

\newcommand{\floor}[1]{\lfloor #1 \rfloor}

\newcommand{\unif}{\operatorname{unif}}

\author{By~~{\scshape Michael Reilly}}
\date{\small \today}
\title{\bfseries Uniform Weighted Averages and a Conjecture of Bergelson, Moreira, and Richter}

\begin{document}
\maketitle



\begin{abstract}
We confirm a conjecture posed by Bergelson, Moreira, and Richter (\cite[Remark 1.12]{Bergelson_Moreira_Richter_2020}), and in particular show that for every probability measure preserving system $(X,\mathscr{B},\mu,T)$, every $k\in \mathbb{N}$, every set $A\in \mathscr{B}$ with $\mu(A)>0$, and every tempered function $f$,
\[
    \lim_{N\to\infty}\frac{1}{N}\sum_{n=1}^N\mu(A\cap T^{-\lfloor{f(n)\rfloor}}A\cap T^{-\lfloor{f(n+1)\rfloor}}A\cap \cdots \cap T^{-\lfloor{f(n+k)\rfloor}}A)>0.
\]

This is achieved by establishing conditions on an increasing function $W:\N\rightarrow (0,\oo)$ such that if $(x_n)_{n\in \N}$ is a bounded sequence in a Banach space with
   \[
   \lim_{W(N)-W(M)\to\oo}\frac{1}{W(N)-W(M)}\sum_{n=M}^N (W(n)-W(n-1))x_n =L
   \]
   then the limit of Ces\`aro averages of $(x_n)_{n\in \N}$, $\lim_{N\to\oo}\frac{1}{N}\sum_{n=1}^Nx_n$ is also equal to $L$.
   
Furthermore, the methods we develop can be used to sharpen some of the combinatorial results obtained in \cite{Bergelson_Moreira_Richter_2020}. For example, if $E$ is a set of positive upper density, then for any $k\in \mathbb{N}$, any $\epsilon>0$,  and all sufficiently large $N\in \N$ there is an $n\in [N-N^{1/2+\epsilon},N]$ such that
    \[\{a,a+\floor{n^{3/2}},a+\floor{(n+1)^{3/2}},\dots ,a +\floor{(n+k)^{3/2}}\}\subseteq E.
    \]
\end{abstract}

\section{Introduction}

Our goal is to resolve a conjecture posed in \cite[Remark 1.12]{Bergelson_Moreira_Richter_2020} by showing that multiple ergodic averages of the form 
\begin{equation}\label{eq:mea}
    \frac{1}{N}\sum_{n=1}^N\mu(A\cap T^{-\floor{f(n)}}A\cap T^{-\lfloor{f(n+1)\rfloor}}A\cap \cdots \cap T^{-\floor{f(n+k)}}A),
\end{equation}
converge to a positive limit as $N\to\oo$ whenever $f$ is a tempered function\footnote{A function $f:[a,\oo)\rightarrow \R$ is called a \emph{tempered function} if there exists $\ell\in \N$ such that $f^{(\ell)}$ decreases to $0$ and $\lim_{x\to\oo}xf^{(\ell)}(x)=\oo$}. We will prove this by developing some new tools regarding weighted averages and applying a theorem from \cite{Bergelson_Moreira_Richter_2020}.

In order to discuss this in more detail, we first require some definitions. Let $\Delta$ denote the \emph{discrete derivative}\footnote{Other sources may define the discrete derivative as $\Delta f(n) = f(n+1)-f(n)$.}, which acts on a function $f$ defined on $\N=\{1,2,\dots\}$ by $\Delta f(n)= f(n)-f(n-1)$ for $n\geq 2$ and $\Delta f(1) = f(1)$. For $\ell\geq 1$, we define $\Delta^{\ell}f(n) = \Delta (\Delta^{\ell-1}f(n))$, where $\Delta ^0f(n) =  f(n)$. For $p(x) = a_mx^m+\dots+a_1x+a_0\in \Z[x]$, we put 
\[
p(\Delta)f(n) = a_m\Delta^mf(n)+\cdots+a_1\Delta f(n)+a_0f(n).
\]

\begin{definition}
Let $\mathscr{F}_1$ be the collection of functions defined by
    \[
    \mathscr{F}_1 = \{f:\N\rightarrow \R: \lim_{n\to\oo}f(n) = \oo, \lim_{n\to\oo}\Delta f(n) = 0,\text{ and } \Delta f\text{ is eventually decreasing}\}.
    \]
    Having defined $\mathscr{F}_{\ell}$ for some $\ell\in \N$, let $\mathscr{F}_{\ell+1} = \{f:\N\rightarrow \R:\Delta f \in \mathscr{F}_{\ell}\}$. Finally, let $\mathscr{F} = \bigcup_{\ell=1}^{\oo}\mathscr{F}_{\ell}$. We may note that $\mathscr{F}$ contains all tempered functions.
\end{definition}
\begin{example} A function $f:\N\rightarrow \R$ is contained in $\mathscr{F}$ if and only if there is an $\ell\in \N$ such that $\Delta^{\ell} f$ is eventually decreasing to $0$ but $\Delta^{\ell-1}f$ tends to $\oo$. So any function of the form $f(n) = n^c$ for $c\in (0,\oo)\setminus \Z$ is contained in $\mathscr{F}$, but no polynomial is contained in $\mathscr{F}$. Let $p(x) = 3x^3-5x^2+x-1\in \Z[x]$ and let $f(x) = x^{\frac{3}{2}}\in \mathscr{F}$. Then 
$
p(\Delta)f(n) = 3(n-3)^{\frac{3}{2}}-5(n-2)^{\frac{3}{2}}+(n-1)^{\frac{3}{2}}-n^{\frac{3}{2}}.
$
\end{example}

The following is Theorem D in \cite{Bergelson_Moreira_Richter_2020}.
\begin{theorem}[{\cite[Theorem D]{Bergelson_Moreira_Richter_2020}}]\label{thm:D}
    Let $k \in \N, \ell\in \N$, $f \in \mathscr{F}_{\ell}$. Let $W= \Delta^{\ell-1}f$ so that $\lim_{n\to\oo}W(n)=\oo$ and $\lim_{n\to\oo}\Delta W(n)=0$.  Let $(X,\mathscr{B},\mu,T)$ be an invertible probability measure preserving system and let $p_1,\dots,p_k \in \Z[x]$. 
    \begin{enumerate}[label = (\alph*)]
        \item For any $h_1,\dots,h_k \in L^{\infty}(X,\mathscr{B},\mu)$, the limit
    \begin{equation*}
       \lim_{W(N)-W(M)\to\oo}\frac{1}{W(N)-W(M)}\sum_{n=M}^N\Delta W(n) T^{\floor{p_1(\Delta)f(n)}}h_1\cdots T^{\floor{p_k(\Delta)f(n)}}h_k 
    \end{equation*}
     exists in $L^2(X,\mathscr{B},\mu)$.
    \item For any $A \in\mathscr{B}$ with $\mu(A) > 0$, the limit
    \begin{equation*}
          \lim_{W(N)-W(M)\to\oo}\frac{1}{W(N)-W(M)}\sum_{n=M}^N\Delta W(n)\mu(A\cap T^{-\floor{p_1(\Delta)f(n)}}A\cap\cdots\cap  T^{-\floor{p_k(\Delta)f(n)}}A)
    \end{equation*}
    exists and is positive.
    \end{enumerate}
\end{theorem}

In \cite[Remark 1.12]{Bergelson_Moreira_Richter_2020}, it is conjectured that each limit of weighted averages appearing in Theorem \ref{thm:D} can be replaced with the corresponding limit of \Cesaro averages whenever $f$ is a tempered function. Theorem \ref{thm:D} applies when $f$ is any element of $\mathscr{F}$, not just when $f$ is a tempered function. However, when $f$ is not tempered the weighted averages in Theorem \ref{thm:D} cannot be replaced by \Cesaro averages (c.f. Example \ref{ex:log_not_ud} where $f(n)=\log(n)$ is not tempered).

The following theorem is a corollary of our main result and it gives an affirmative answer to the conjecture in \cite{Bergelson_Moreira_Richter_2020}.
\begin{maintheorem}\label{thm:A}
    Let $f\in \mathscr{F}_{\ell}$ be a tempered function. Let $W=\Delta^{\ell-1}f$ so that $W$ increases to $\oo$, $\Delta W$ decreases to $0$, and $\lim_{N\to\oo}N\cdot \Delta W(N) = \oo$. Suppose that $(x_n)_{n\in \N}$ is a bounded sequence in a Banach space $Y$ and let $L\in Y$. If 
    \[
      \lim_{W(N)-W(M)\to\oo}\frac{1}{W(N)-W(M)}\sum_{n=M}^N \Delta W(n)x_n = L
    \]
    then 
    \[
    \lim_{N\to\oo}\frac{1}{N}\sum_{n=1}^N x_n=L.
    \]
\end{maintheorem}

Thus, in light of Theorem \ref{thm:A}, the following theorem follows immediately from Theorem \ref{thm:D}.

\begin{theorem}\label{thm:BMR_conjecture}
    Let $(X,\mathscr{B},\mu,T)$ be an invertible probability measure preserving system. Let $f\in \mathscr{F}$ be a tempered function, let $k\in \N$, and let $p_1,\dots,p_k \in \Z[x]$. 
\begin{enumerate}[label = (\alph*)]
        \item For any $h_1,\dots,h_k \in L^{\infty}(X,\mathscr{B},\mu)$, the limit
    \begin{equation}
       \lim_{N\to\oo}\frac{1}{N}\sum_{n=1}^NT^{\floor{p_1(\Delta)f(n)}}h_1\cdots T^{\floor{p_k(\Delta)f(n)}}h_k 
    \end{equation}
     exists in $L^2(X,\mathscr{B},\mu)$.
    \item For any $A \in\mathscr{B}$ with $\mu(A) > 0$, the limit
    \begin{equation}
          \lim_{N\to\oo}\frac{1}{N}\sum_{n=1}^N\mu(A\cap T^{-\floor{p_1(\Delta)f(n)}}A\cap\cdots \cap T^{-\floor{p_k(\Delta)f(n)}}A)
    \end{equation}
    exists and is positive.
    \end{enumerate}
\end{theorem}
Using the identity $f(n-k) = (1+\Delta)^kf(n)$, we obtain a special case of Theorem \ref{thm:BMR_conjecture}, which extends not only Theorem \ref{thm:D} but is also an extension of \cite[Theorem 5.6]{Fran22}. 
\begin{corollary}
    Let $(X,\mathscr{B},\mu,T)$ be an invertible probability measure preserving system. Let $f\in \mathscr{F}$ be a tempered function and let $k\in \N$.
\begin{enumerate}[label = (\alph*)]
        \item For any $h_1,\dots,h_k \in L^{\infty}(X,\mathscr{B},\mu)$, the limit
    \begin{equation}
       \lim_{N\to\oo}\frac{1}{N}\sum_{n=1}^NT^{\floor{f(n)}}h_1\cdots T^{\floor{f(n+k)}}h_k 
    \end{equation}
     exists in $L^2(X,\mathscr{B},\mu)$.
    \item For any $A \in\mathscr{B}$ with $\mu(A) > 0$, the limit
    \begin{equation}
          \lim_{N\to\oo}\frac{1}{N}\sum_{n=1}^N\mu(A\cap T^{-\floor{f(n)}}A\cap\cdots\cap T^{-\floor{f(n+k)}}A)
    \end{equation}
    exists and is positive.
    \end{enumerate}
\end{corollary}

Now we will introduce some notation in order to formulate our main theorem.

\begin{definition}
Let $W:\N\rightarrow \R$ be a function which increases to $\infty$ and let $(x_n)_{n\in \N}$ be a bounded sequence in a Banach space $Y$. 
\begin{itemize}
    \item For $N\in \N$, define the \emph{N-th $W$-weighted average} of $(x_n)_{n\in \N}$ by
\begin{equation}
    \E_{n\leq N}^Wx_n= \frac{1}{W(N)}\sum_{n=1}^N\Delta W(n) x_n.
\end{equation}
\item More generally, we define the $W$-weighted average of $(x_n)_{n\in \N}$ over the interval $[M,N]$ to be
\begin{equation}
      \E_{n\in [M,N]}^Wx_n = \frac{1}{W(N)-W(M)}\sum_{n=M}^N\Delta W(n) x_n.
\end{equation}
\item Define the \emph{uniform $W$-weighted average} of $(x_n)_{n\in \N}$ as 
\begin{equation}
     \E_{\text{unif}}^W(x_n) =\lim_{W(N)-W(M)\to \oo}\E_{n\in [M,N] }^W(x_n) 
\end{equation}
when the limit exists. The limit $\lim_{W(N)-W(M)\to \oo}$ is taken over all sequences of intervals $[M_j,N_j]$ for $j\in \N$, with $W(N_j)-W(M_j)\to \oo$ as $j\to\oo$.
\item Define the \emph{iterated $W$-weighted averages} of $(x_n)_{n\in \N}$ by $\E(1)_{n\leq N}^Wx_n = \E_{n\leq N}^Wx_n$ and
    \begin{equation} 
    \E(k+1)_{n\leq N}^Wx_n = \E_{n\leq N}^W(\E(k)_{m\leq n}^Wx_m) = \frac{1}{W(N)}\sum_{n=1}^N\Delta W(n) \cdot\E(k)_{m\leq n}^Wx_m
    \end{equation}
    for $k\in \N$.
\item For $L\in Y$, we say that $\E(\oo)^Wx_n = L$ if
    \begin{equation} 
\lim_{k\to\oo}\limsup_{N\to\oo}|\E(k)_{n\leq N}^W(x_n -L)|=0.
    \end{equation}
\end{itemize}
\end{definition}
\begin{example}
    When $W(N) = N$, $\E_{n\leq N}^Wx_n$ denotes the usual \Cesaro averages and we will write $\E_{n\in [M,N]}x_n$ instead of $\E^W_{n\in [M,N]}x_n$ so that we have $\E_{n\in [M,N]}x_n =   \E_{n\in [M,N]}^Wx_n = \frac{1}{N-M}\sum_{n=M}^{N}x_n$.
    When $U(N) = \log(N)$, $\E_{n\leq N}^Ux_n$ is the logarithmic average
    \begin{equation}
         \E_{n\leq N}^Ux_n = \frac{1}{\log N}\sum_{n=1}^N\Delta \log(n) x_n =  \frac{1}{\log N}\sum_{n=1}^N\frac{x_n}{n} + o_{N\to\oo}(1).
    \end{equation}

    It will be useful to consider weighted averages with respect to functions which grow faster than any polynomial. For example, let $V(x) = e^{\sqrt{x}}$. Then
    \begin{equation}
         \E_{n\leq N}^Vx_n = \frac{1}{e^{\sqrt{N}}}\sum_{n=1}^N\Delta (e^{\sqrt{n}})\cdot x_n =  \frac{1}{e^{\sqrt{N}}}\sum_{n=1}^N\frac{e^{\sqrt{n}}x_n}{2\sqrt{n}}  + o_{N\to\oo}(1).
    \end{equation}
\end{example}

\begin{example}
    When $W(N) = N$, the uniform $W$-weighted average of $(x_n)_{n\in \N}$ is the usual notion of the uniform \Cesaro limit
    \begin{equation}
       \lim_{N-M\to\oo} \frac{1}{N-M}\sum_{n=M}^Nx_n
    \end{equation}
    and $\E_{\unif}x_n = L$ is equivalent to the statement that $\lim_{N\to\oo}\frac{1}{N}\sum_{n=k}^{N+k}x_n$ converges to $L$ uniformly in $k$.

    When $U(N) = \log(N)$, the uniform $U$-weighted average of $(x_n)_{n\in\N}$ is 
    \begin{equation}
         \lim_{\log(N)-\log(M)\to\oo}\E_{n\in [M,N]}^U(x_n) = \lim_{N/M\to\oo}\frac{1}{\log \frac{N}{M}}\sum_{n=M}^N\Delta \log(n)\cdot x_n,
    \end{equation}
and $\E_{\unif}^{\log}x_n =L$ is equivalent to the statement that $\lim_{a\to\oo}\frac{1}{\log a}
\sum_{n=k}^{ak}\frac{x_n}{n}$ converges to $L$ uniformly in $k$.
\end{example}

We now show that Theorem \ref{thm:A} follows from our main result, Theorem \ref{thm:general_main}.

\begin{maintheorem}\label{thm:general_main}
 Let $V:\N\rightarrow (0,\oo)$ be a function which increases to $\oo$ with $\Delta \log(V(N))$ eventually decreasing to $0$ such that $\lim_{N\to\oo}\frac{\Delta \log(V(N))}{\Delta \log(N)}=\oo$. Let $Y$ be a Banach space, let $(x_n)_{n\in \N}\subset Y$ be a bounded sequence, and let $L\in Y$.
    Consider the following statements:
   \begin{enumerate}[label = (\arabic*)]
   \item $\lim_{N\to\oo}\E_{n\leq N}^Vx_n =L$,
         \item $\E(\oo)^Vx_n = L$,
       \item $\lim_{N\to\oo}\E_{n\leq N}^Ux_n = L$ for each function $U:\N\rightarrow (0,\oo)$ which increases to $\oo$ and satisfies
       \begin{equation}
       \label{eq:U_condition}
       \lim_{N\to\oo}\frac{V(N-1)}{\Delta U(N)}\left(\frac{\Delta U(N)}{\Delta V(N)}-\frac{\Delta U(N-1)}{\Delta V(N-1)}\right)=-1,
       \end{equation}
       \item $\lim_{N\to\oo} \E_{n\in [N-s(N),N]}x_n = L$ for each nondecreasing function $s:\N\rightarrow\N$ satisfying 
       \begin{equation} \label{eq:s_condition}
       \lim_{N\to\oo}s(N) \cdot \Delta \log(V(N))=\oo
       \end{equation}
       and $s(N)\leq N-1$ for all $N\in \N$,
        \item $\E_{\text{unif}}^{\log V}x_n = L$.
        
   \end{enumerate}
  We have that  
  \[
  (1)\implies (2)\iff (3) \iff (4)\iff (5).
  \]
\end{maintheorem}
\begin{proof}[Proof of Theorem \ref{thm:A}]
Let $V(N) = e^{W(N)}$ and suppose that 
\[
  \lim_{W(N)-W(M)\to\oo}\frac{1}{W(N)-W(M)}\sum_{n=M}^N \Delta W(n)x_n = L,
\]
i.e. $\E_{\unif}^Wx_n = L$ and so condition (5) of Theorem \ref{thm:general_main} holds. We have that $\Delta \log(V(N))$ is decreasing to $0$ and that
\[
\lim_{N\to\oo}\frac{\Delta \log(V(N))}{\Delta \log(N)} =\lim_{N\to\oo}N\cdot \Delta W(N) = \oo
\]
because of our assumptions on $W$. Thus we can apply Theorem \ref{thm:general_main} to deduce that condition (4) holds. 

Let $s(N)  = N-1$ so that $\E_{n\in [N-s(N),N]}x_n = \E_{n\in [1,N]}x_n = \E_{n\leq N}x_n$ for all $N\in \N$ and $\lim_{N\to\oo}s(N)\cdot \Delta \log(V(N)) = \lim_{N\to\oo}N\cdot \Delta W(N)=\oo$. By Theorem \ref{thm:general_main}, we have that $\E_{n\in [N-s(N),N]}x_n=L$. This completes the proof.
\end{proof}
\begin{remark}\label{remark:hardy_functions}
    The limit in equation (\ref{eq:U_condition}) initially appears unwieldy and impractical. However, let us assume that the limit in equation (\ref{eq:U_condition}) exists. Then we can perform some algebraic manipulations and apply the Stolz-\Cesaro Theorem \cite[Problem 70]{Polya_Szego} to see that 
    \begin{align*}
        &1+\lim_{N\to\oo}\frac{V(N-1)}{\Delta U(N)}\left(\frac{\Delta U(N)}{\Delta V(N)}-\frac{\Delta U(N-1)}{\Delta V(N-1)}\right)\\
        =&\lim_{N\to\oo}
        \frac{
        \Delta V(N)\cdot \frac{ \Delta U(N)}{\Delta V(N)}+V(N-1)\cdot \Delta \left(\frac{ \Delta U(N)}{\Delta V(N)}\right)}
        {\Delta U(N)}\\
        =&\lim_{N\to\oo}\frac{\Delta \left(V(N)\cdot\frac{ \Delta U(N)}{\Delta V(N)}\right)}{\Delta U(N)}=\lim_{N\to\oo}\frac{V(N)\cdot \Delta U(N)}{U(N)\cdot \Delta V(N)}.
    \end{align*}
So the limit in equation (\ref{eq:U_condition}) is equal to $-1$ if and only if $\lim_{N\to\oo}\frac{V(N)\cdot \Delta U(N)}{U(N)\cdot \Delta V(N)}=0$. We can obtain a clearer criterion when we use our assumption that $\lim_{N\to\oo}\Delta \log(V(N))=0$, since we have
\begin{equation}\label{eq:log_and_derivaitve}
-\Delta \log V(N) = \log\left(\frac{V(N-1)}{V(N)}\right) = \log\left(1-\frac{\Delta V(N)}{V(N)}\right)= \frac{\Delta V(N)}{V(N)}+\left(\frac{\Delta V(N)}{V(N)}\right)^2+\cdots
\end{equation}
and so $\lim_{N\to\oo}\Delta \log V(N) = 0$ if and only if $\lim_{N\to\oo}\frac{\Delta V(N)}{V(N)}=0$.
It follows that 
\[
\lim_{N\to\oo}\frac{V(N)\cdot \Delta U(N)}{U(N)\cdot \Delta V(N)}=0 \quad \text{ if and only if }\quad  \lim_{N\to\oo}\frac{\Delta \log(U(N))}{\Delta \log(V(N))}=0.
\]
One more application of the Stolz-\Cesaro Theorem shows that if $\lim_{N\to\oo}\frac{\Delta \log(U(N))}{\Delta \log(V(N))}=0$ then $\lim_{N\to\oo}\frac{ \log(U(N))}{ \log(V(N))}=0$. In particular, when the limit in equation (\ref{eq:U_condition}) is known to exist, say when $U$ and $V$ are contained in the same Hardy\footnote{A Hardy field is a field of real valued functions which is closed under derivation, under the equivalence relation that two functions are equal if they are equal outside of a compact set. What is important for our purposes is any function contained in a Hardy field is eventually monotone and so any limit involving functions contained in the same Hardy field will exist in $\R\cup\{-\oo,+\oo\}$. For more details, see \cite{Boshernitzan}.} field, the assumption that equation (\ref{eq:U_condition}) holds can be replaced with the assumption that $\lim_{N\to\oo}\frac{\log(U(N))}{\log(V(N))}=0$.
\end{remark}
Remark \ref{remark:hardy_functions} gives a special case of Theorem \ref{thm:general_main}.
\begin{maintheorem}\label{thm:general_main_Hardy}
 Let $V$ be a function which is contained in a Hardy field and tends to $\oo$. Suppose that $\lim_{N\to\oo}\frac{ \log(V(N))}{ \log(N)}=\oo$ and $\lim_{N\to\oo}\frac{ \log(V(N))}{ N}=0$. Let $Y$ be a Banach space, let $(x_n)_{n\in \N}\subset Y$ be a bounded sequence, and let $L\in Y$.
    Consider the following statements.
   \begin{enumerate}[label = (\arabic*)]
   \item $\lim_{N\to\oo}\E_{n\leq N}^Vx_n =L$,
         \item $\E(\oo)^Vx_n = L$,
       \item $\lim_{N\to\oo}\E_{n\leq N}^Ux_n = L$ for each function $U$ which is contained in the same Hardy field as $V$, tends to $\oo$, and satisfies
       $
       \lim_{N\to\oo}\frac{\log(U(N))}{\log(V(N))}=0,
      $
       \item $\lim_{N\to\oo} \E_{n\in [N-s(N),N]}x_n = L$ for each nondecreasing function $s:\N\rightarrow\N$ satisfying 
       $
       \lim_{N\to\oo}s(N) \cdot \Delta \log(V(N))=\oo
       $
       and $s(N)\leq N-1$ for all $N\in \N$,
        \item $\E_{\text{unif}}^{\log V}x_n = L$.
        
   \end{enumerate}
  We have that  
  \[
  (1)\implies (2)\iff (3) \iff (4)\iff (5).
  \]
\end{maintheorem}

Of special interest in Theorem \ref{thm:general_main} is condition (4), since this allows us to relate weighted averages to \Cesaro averages along a sequence of intervals.

\begin{example}
    Let $E\subseteq \N$ be a set with 
    \[
    \limsup_{N-M\to\oo}\frac{|E\cap [M,N]|}{N-M}>0.
    \]
     Let $f\in\mathscr{F}_{\ell}$ and let $k\in \N$. Let $W= \Delta^{\ell-1}f$. For each $n\in \N$, let $A(n)$ be the set of $a\in \N$ such that 
\[
\{a,a+\floor{f(n)}, a+\floor{f(n+1)},\dots, a+\floor{f(n+k)}\}\subset E.
\]
From \cite[Corollary F]{Bergelson_Moreira_Richter_2020} we know that the set 
\[
S = \left\{n\in \N: \limsup_{N-M\to\oo}\frac{|A(n)\cap [M,N]|}{N-M}>0\right\}
\]
is \emph{$W$-syndetic}, meaning that 
\begin{equation}\label{eq:w_syndetic}
\limsup_{W(N)-W(M)\to\oo}\frac{|S\cap [M,N]|}{W(N)-W(M)}>0.
\end{equation}
Using Theorem \ref{thm:general_main}, we may rephrase this $W$-syndetic condition as saying that 
\begin{equation}\label{eq:s_W_syndetic}
\limsup_{N\to\oo}\E_{n\in [N-s(N),N]}1_{S}(n) >0
\end{equation}
for any nondecreasing function $s\colon \N\rightarrow \N$ with $\lim_{N\to\oo}s(N)\cdot \Delta W(N)=\oo$ and $s(N)\leq N-1$ for all $N\in \N$.     
More concretely, for the sake of example we can take $f(n) = n^{3/2}$. Fix $k\in \N$ and let $\epsilon>0$. We can take $s(N) = N^{1/2+\epsilon}$ in equation (\ref{eq:s_W_syndetic}) to see that for all large enough $N$, there is an $n\in [N-N^{1/2+\epsilon},N]$ such that $E$ contains a configuration of the form  
    \begin{equation*}
    \{a,a+\floor{n^{3/2}},a+\floor{(n+1)^{3/2}},\dots ,a +\floor{(n+k)^{3/2}}\}
    \end{equation*}
    and moreover, these configurations are abundant in the sense of equation (\ref{eq:w_syndetic}).
\end{example}

\subsection*{Acknowledgments}
The creation of this paper would not have been possible without the generous advice and direction from the author's advisor, Vitaly Bergelson. We would also like to thank Nikos Frantzikinakis for helpful conversation.

\subsection*{Structure of the paper}
The remainder of the paper consists of a proof of Theorem \ref{thm:general_main}.
It is clear that $(1)\implies (2)$ by definition. We will prove the rest of the implications one by one. In Section \ref{section:2_implies_3} we prove $(2)\implies (3)$, in Section \ref{section:3_implies_4} we prove $(3)\implies (4)$, and in Section \ref{section:4_implies_5} we prove $(4)\implies (5)$. Finally, in Section \ref{section:5_implies_2} we prove that $(5)\implies (2)$. 

\section[Section 2]{$(2)\implies (3)$}\label{section:2_implies_3}

In this section we prove the implication $(2)\implies (3)$ in Theorem \ref{thm:general_main}. This implication will follow readily from the following lemma, whose proof appears in an unpublished manuscript authored by Michael Boshernitzan \cite{Boshernitzan87}.
\begin{lemma}\label{thm:slower_wins}
   Let $V,U:\N\rightarrow (0,\oo)$ increase to $\oo$ and suppose that 
   \[
   \lim_{N\to\oo}\frac{V(N-1)}{\Delta U(N)}\left(\frac{\Delta U(N)}{\Delta V(N)}-\frac{\Delta U(N-1)}{\Delta V(N-1)}\right)=-1.
   \]
   Additionally assume that $\limsup_{N\to\oo}\frac{\Delta V(N)}{V(N)}<1$. Let $Y$ be a Banach space and let $(x_n)_{n\in \N}\subseteq Y$ be a bounded sequence. Then
   \begin{equation}\label{eq:boshernitzan}
        \E_{n\leq N}^U(\E_{k\leq n}^Vx_k) = \E_{n\leq N}^Ux_n+o_{N\to\oo}(1),
    \end{equation}
    where the $o_{N\to\oo}(1)$ term depends only on $U,V$ and $\sup_{n\in \N}\|x_n\|$.
\end{lemma}
\begin{proof}

Recall that the summation by parts formula says that
\begin{align}\label{eq:summation_by_parts}
        \sum_{n=1}^N\Delta f(n)\cdot g(n) = f(N)\cdot g(N)-\sum_{n=1}^{N-1}f(n)\cdot \Delta g(n+1)
\end{align}
for functions $f,g$ on $\N$. By the hypothesis of the lemma, we may apply the results of Remark \ref{remark:hardy_functions} to obtain $\lim_{N\to\oo}\frac{V(N)\cdot \Delta U(N)}{U(N)\cdot \Delta V(N)}=0$ and hence $\lim_{N\to\oo}\frac{\Delta U(N)}{U(N)}=0$.

We will take the right-hand side of equation (\ref{eq:boshernitzan}) and rewrite it in the form of the left-hand side. Put $F(n) = \sum_{k=1}^n\Delta V(k)\cdot x_k$ so that $\Delta F(n) = \Delta V(n)\cdot x_n$ and $\frac{F(N)}{V(N)}=\E_{n\leq N}^Vx_n$. 

Apply summation by parts to obtain
    \begin{align}
        &\E^U_{n\leq N}x_n =\frac{1}{U(N)}\sum_{n=1}^N\Delta U(n)\cdot x_n =\frac{1}{U(N)}\sum_{n=1}^N\frac{\Delta U(n)}{\Delta V(n)}\cdot \Delta F(n)\\
        =& \frac{1}{U(N)}\left(\frac{\Delta U(N)}{\Delta V(N)}\cdot F(N)- \sum_{n=1}^{N-1}F(n)\cdot \Delta\left(\frac{\Delta U(n+1)}{\Delta V(n+1)}\right) \right)\\
        =& \frac{\Delta U(N)}{U(N)\cdot \Delta V(N)}\cdot F(N)- \frac{1}{U(N)}\sum_{n=1}^{N-1}F(n)\cdot \Delta\left(\frac{\Delta U(n+1)}{\Delta V(n+1)}\right)\label{eq:sum_by_parts}.
        %
    \end{align}
    The first term in (\ref{eq:sum_by_parts}) is equal to $\frac{V(N)\cdot \Delta U(N)}{U(N)\cdot \Delta V(N)}\cdot \frac{F(N)}{V(N)}$, which tends to $0$ as $N\to\oo$ since $\frac{F(N)}{V(N)} = \E_{n\leq N}^Vx_n$ is bounded in norm by $\sup_{n\leq N}\|x_n\|$. The second term in (\ref{eq:sum_by_parts}) is    \begin{align*}\label{eq:sum_by_parts_second_term}
        -& \frac{1}{U(N)}\sum_{n=1}^{N-1}F(n)\cdot \Delta\left(\frac{\Delta U(n+1)}{\Delta V(n+1)}\right)=- \frac{1}{U(N)}\sum_{n=1}^{N-1}F(n)\cdot \left(\frac{\Delta U(n+1)}{\Delta V(n+1)}-\frac{\Delta U(n)}{\Delta V(n)}\right)
        \\
        = -&\frac{1}{U(N)}\sum_{n=1}^{N-1}\Delta U(n+1)\cdot \frac{F(n)}{V(n)}\cdot \frac{V(n)}{\Delta U(n+1)}\left(\frac{\Delta U(n+1)}{\Delta V(n+1)}-\frac{\Delta U(n)}{\Delta V(n)}\right)\\
        =-&\frac{1}{U(N)}\sum_{n=2}^{N}\Delta U(n)\cdot \frac{F(n-1)}{V(n-1)}\cdot \frac{V(n-1)}{\Delta U(n)}\left(\frac{\Delta U(n)}{\Delta V(n)}-\frac{\Delta U(n-1)}{\Delta V(n-1)}\right).
    \end{align*}
    By assumption, we have $\frac{V(n-1)}{\Delta U(n)}\left(\frac{\Delta U(n)}{\Delta V(n)}-\frac{\Delta U(n-1)}{\Delta V(n-1)}\right) = -1+o_{N\to\oo}(1)$.
    

   Note that $\frac{F(n-1)}{V(n-1)} = \E_{k\leq n-1}^Vx_k$ is bounded in norm by $\sup_{n\in \N}\|x_n\|$. So (\ref{eq:sum_by_parts}) is 
\begin{align*}
 - &\frac{1}{U(N)}\sum_{n=2}^{N}\Delta U(n)\cdot \E_{k\leq n-1}^Vx_k\cdot (-1+o_{N\to\oo}(1))\\
 =&\E_{n\leq N}^U(\E_{k\leq n-1}^Vx_k)+o_{N\to\oo}(1)
\end{align*}
and the $o_{N\to\oo}(1)$ term will depend only on $U,V$ and $\sup_{n\leq N}\|x_n\|$, and in particular not on $(x_n)_{n\in \N}$. Finally, recalling that $\lim_{N\to\oo}\frac{\Delta U(N)}{U(N)}=0$, we have 
\[
\E_{n\leq N}^U(\E_{k\leq n-1}^Vx_k)=\E_{n\leq N}^U(\E_{k\leq n}^Vx_k)+o_{N\to\oo}(1).
\]

\end{proof}

\begin{corollary}
    Let $Y$ be a Banach space, let $(x_n)_{n\in \N}\subseteq Y$ be a bounded sequence, and let $L\in Y$. Suppose that $V:\N\rightarrow (0,\oo)$ is a function which increases to $\oo$ with $\limsup_{N\to\oo}\frac{\Delta V(N)}{V(N)}<1$. Let $U:\N\rightarrow (0,\oo)$ be a function which increases to $\oo$ and satisfies
       \begin{equation*}
       \lim_{N\to\oo}\frac{V(N-1)}{\Delta U(N)}\left(\frac{\Delta U(N)}{\Delta V(N)}-\frac{\Delta U(N-1)}{\Delta V(N-1)}\right)=-1.
       \end{equation*}
    If $\E(\oo)^Vx_n = L$ then 
 $\lim_{N\to\oo}\E_{n\leq N}^Ux_n = L$.
\end{corollary}

\begin{proof}

Suppose that $\E(\oo)^Vx_n = L$. Let $\epsilon>0$. We can pick $K\in \N$ such that $|\E_{n\leq N}^V(k)(x_n)-L|<\epsilon$ for all large enough $N$. Now consider $\E_{n\leq N}^U(\E_{k\leq n}^V(K)x_k)$. From equation (\ref{eq:boshernitzan}), 
\begin{align*}
\E_{n\leq N}^U(\E_{k\leq n}^V(K)x_k) = &\E_{n\leq N}^U(\E_{k\leq n}^V(K-1)x_k) +o_{N\to\oo}(1)
\\
= &\E_{n\leq N}^U(\E_{k\leq n}^V(K-2)x_k) +o_{N\to\oo}(1)\\
\vdots\\
=&\E_{n\leq N}^U(\E_{k\leq n}^Vx_k) +o_{N\to\oo}(1)\\
=&\E_{n\leq N}^Ux_n +o_{N\to\oo}(1).
\end{align*}

Using linearity, we have
\begin{equation}\label{eq:K_composition}
\E_{n\leq N}^U(\E_{k\leq n}^V(K)(x_k-L)) = \E_{n\leq N}^Ux_n -L+o_{N\to\oo}(1).
\end{equation}
But we can bound $|\E_{n\leq N}^U(\E_{k\leq n}^V(K)x_k-L)|$ using the triangle inequality,
\begin{align*}
|\E_{n\leq N}^U(\E_{k\leq n}^V(K)x_k-L)| \leq& \E_{n\leq N}^U|\E_{k\leq n}^V(K)(x_k)-L|\\
=& \E_{n\leq N}^U(\epsilon+o_{n\to\oo}(1))\\
=& \epsilon + o_{N\to\oo}(1).
\end{align*}
Taking $\limsup_{N\to\oo}$ of equation (\ref{eq:K_composition}), we find that
\[
\limsup_{N\to\oo}|\E_{n\leq N}^Ux_n -L| \leq \epsilon+o_{N\to\oo}(1).
\]
This holds for each $\epsilon>0$, so we conclude that $\lim_{N\to\oo}\E_{n\leq N}^Ux_n=L$.

\end{proof}

\section[Section 3]{$(3)\implies(4)$}\label{section:3_implies_4}

In this section we prove the implication $(3)\implies (4)$ in Theorem \ref{thm:general_main}. We begin by recalling the Silverman-Toeplitz Theorem \cite[Theorem 2.3.7]{Boos}.

\begin{theorem}[Silverman–Toeplitz]\label{thm:silverman_toeplitz}
    Let $(c_{N,n})_{N,n\in \N}$ be a doubly indexed sequence of complex numbers and let $Y$ be a Banach space. The following are equivalent.
    \begin{enumerate}[label = (\Alph*)]
        \item $\lim_{N\to\oo}\sum_{n\in \N}c_{N,n}y_n = \lim_{n\to\oo}y_n$ for each convergent sequence $(y_n)_{n\in \N}\subseteq Y$.
        \item $(c_{N,n})_{N,n\in \N}$ satisfies each of the following:
        \begin{itemize} 
        \item For each $n\in \N$, $\lim_{N\to\oo}c_{N,n}=0$,
        \item $\lim_{N\to\oo}\sum_{n\in \N}c_{N,n}=1$,
        \item $\limsup_{N\to\oo}\sum_{n\in \N}|c_{N,n}|<\oo$.
        \end{itemize}
    \end{enumerate}
\end{theorem}

The Silverman-Toeplitz theorem will allow us to deduce the desired implication once we show that there is a sequence $(c_{N,n})_{N,n\in \N}$ relating the averages of the form $\E_{n\in [N-s(N),N]}$ and the averages of the form $\E_{n\leq N}^U$.

\begin{lemma} \label{lem:weighted_implies_interval}
Let $Y$ be a Banach space, let $U:\N\rightarrow (0,\oo)$ be an increasing function which satisfies $\limsup_{N\to\oo}\frac{\Delta \log(U(N))}{\Delta \log(N)}=\oo$ and $\limsup_{N\to\oo}\frac{\Delta U(N)}{U(N)}=0$.
Additionally, assume that $\Delta U$ is nondecreasing.
Let $s:\N\rightarrow \N$ be a nondecreasing function with $s(N)\leq N-1$ for all $N\in \N$ and 
\[
\lim_{N\to\oo}s(N)\cdot \Delta \log U(N)=\oo.
\]

Then there exists a doubly indexed nonnegative sequence $(c_{N,n})_{N,n\in \N}$ which satisfies each of the conditions in item (B) of Theorem \ref{thm:silverman_toeplitz} and such that for any bounded sequence $(x_n)_{n\in \N}\subset Y$,
\begin{equation}\label{eq:weighted_implies_interval}
\E_{n\in [N-s(N),N]}x_n = \sum_{n=1}^Nc_{N,n}\E^U_{k\leq n}x_k+o_{N\to\oo}(1).
\end{equation}
\end{lemma}

\begin{proof}

Define $(c_{N,n})_{N,n\in \N}$ by 
\[
c_{N,n} = 
\begin{cases}
     \frac{1}{s(N)}\cdot U(N)\cdot \frac{1}{\Delta U(N)} &\text{ if } n=N\\
    \frac{1}{s(N)}\cdot U(n)\cdot \left(\frac{1}{\Delta U(n)}-\frac{1}{\Delta U(n+1)}\right) &\text{ if } N-s(N)\leq n\leq N-1\\
    0  &\text{ otherwise }\\
\end{cases}
\]
Each $c_{N,n}$ is nonnegative since $\Delta U$ is nondecreasing and it is clear that $\lim_{N\to\oo}c_{N,n}=0$ for each fixed $n$ since $s(N)\to \oo$ as $N\to\oo$. Next we will show that equation (\ref{eq:weighted_implies_interval}) holds. Switching the order of summation, we have
\begin{align*} 
&\sum_{n=1}^Nc_{N,n}\E^U_{k\leq n}x_k= \sum_{n=1}^Nc_{N,n}\frac{1}{U(n)}\sum_{k=1}^n\Delta U(k) x_k=\sum_{k=1}^Nx_k\cdot \left(\Delta U(k)\sum_{n=k}^N\frac{c_{N,n}}{U(n)}\right).
\end{align*}
We will show the following two equations, from which equation (\ref{eq:weighted_implies_interval}) follows,
\begin{equation}\label{eq:w_implies_i_1}
    \Delta U(k)\sum_{n=k}^N\frac{c_{N,n}}{U(n)}=\frac{1}{s(N)}\text{ for }N-s(N)\leq k\leq N
\end{equation}
and 
\begin{equation}\label{eq:w_implies_i_2}
    \left|\sum_{k=1}^{N-s(N)-1}x_k\cdot \left(\Delta U(k)\sum_{n=k}^N\frac{c_{N,n}}{U(n)}\right) \right|\to 0 \text{ as }N\to\oo.
\end{equation}

For (\ref{eq:w_implies_i_1}), if $k=N$ we have
\[
 \Delta U(k)\sum_{n=k}^N\frac{c_{N,n}}{U(n)} =  \Delta U(N)\frac{c_{N,N}}{U(N)}  = \Delta U(N)\frac{\frac{1}{s(N)}\cdot U(N)\cdot \frac{1}{\Delta U(N)}}{U(N)}= \frac{1}{s(N)}, 
\]
and if $N-s(N)\leq k<N$ then 
\begin{align*}
     &\Delta U(k)\sum_{n=k}^N\frac{c_{N,n}}{U(n)}  = \Delta U(k)\cdot \frac{c_{N,N}}{U(N)}+\Delta U(k)\sum_{n=k}^{N-1}\frac{c_{N,n}}{U(n)} \\
     =& \Delta U(k)\cdot \frac{ \frac{1}{s(N)}\cdot U(N)\cdot \frac{1}{\Delta U(N)}}{U(N)}+\Delta U(k)\sum_{n=k}^{N-1}\frac{ \frac{1}{s(N)}\cdot U(n)\cdot \left(\frac{1}{\Delta U(n)}-\frac{1}{\Delta U(n+1)}\right)}{U(n)} \\
     =& \frac{\Delta U(k)}{s(N)}\left(\frac{1}{\Delta U(N)}+\sum_{n=k}^{N-1} \left(\frac{1}{\Delta U(n)}-\frac{1}{\Delta U(n+1)}\right)\right)=\frac{\Delta U(k)}{s(N)}\cdot \frac{1}{\Delta U(k)} = \frac{1}{s(N)}.
\end{align*}

Turning our attention to (\ref{eq:w_implies_i_2}), we can observe that $c_{N,k} = 0$ for $k<N-s(N)$ and so as above we have
\begin{align*}
&\Delta U(k)\sum_{n=k}^N\frac{c_{N,n}}{U(n)} =  \Delta U(k)\cdot\sum_{n=N-s(N)}^N\frac{c_{N,n}}{U(n)} \\
=&  \frac{\Delta U(k)}{s(N)}\left(\frac{1}{\Delta U(N)}+\sum_{n=N-s(N)}^{N-1} \left(\frac{1}{\Delta U(n)}-\frac{1}{\Delta U(n+1)}\right)\right)\\
=&   \frac{\Delta U(k)}{s(N)}\cdot \frac{1}{\Delta U(N-s(N))}.
\end{align*}

Then 
\begin{align}
&\sum_{k=1}^{N-s(N)-1}x_k\cdot \left(\Delta U(k)\sum_{n=k}^N\frac{c_{N,n}}{U(n)}\right) = \sum_{k=1}^{N-s(N)-1}x_k\cdot  \frac{\Delta U(k)}{s(N)}\cdot \frac{1}{\Delta U(N-s(N))} \\
=& \frac{1}{s(N)\cdot \Delta \log U(N)}\cdot\frac{\Delta \log U(N)}{\frac{\Delta U(N)}{U(N)}}\cdot \frac{U(N-s(N)-1)}{U(N)}\cdot \E_{n\leq N-s(N)-1}^Ux_n.\label{eq:tail_goes_to_0}
\end{align}
Recall that $\E_{n\leq N-s(N)-1}^Ux_n$ is bounded and that $\frac{U(N-s(N)-1)}{U(N)}\leq 1$ since $U$ is increasing. Additionally, from equation (\ref{eq:log_and_derivaitve}), $
\lim_{N\to\oo}\frac{\Delta \log U(N)}{\frac{\Delta U(N)}{U(N)}} =1$. By assumption, we have $\lim_{N\to\oo}s(N)\cdot \Delta \log U(N)=\oo$ and so taking the limit of equation (\ref{eq:tail_goes_to_0}) shows that equation (\ref{eq:w_implies_i_2}) holds. Lastly, we have $\lim_{N\to\oo}\sum_{n\in \N}c_{N,n} = 1$ by taking $x_n = 1$ for all $n\in \N$ in equation (\ref{eq:weighted_implies_interval}). This concludes the proof.
\end{proof}

\begin{corollary}
\label{cor:weighted_implies_interval}
  Let $Y$ be a Banach space, let $(x_n)_{n\in \N}\subset Y$ be a bounded sequence, and let $L\in Y$. Let $V:\N\rightarrow (0,\oo)$ be an increasing function such that $\Delta \log(V(N))$ decreases to $0$ with $\lim_{N\to\oo}\frac{\Delta \log V(N)}{\Delta \log(N)}=\oo$. Suppose that $\lim_{N\to\oo}\E_{n\leq N}^Ux_n = L$ for each function $U:\N\rightarrow (0,\oo)$ which increases to $\oo$ and satisfies
       \begin{equation*}\tag{\ref{eq:U_condition}}
       \lim_{N\to\oo}\frac{V(N-1)}{\Delta U(N)}\left(\frac{\Delta U(N)}{\Delta V(N)}-\frac{\Delta U(N-1)}{\Delta V(N-1)}\right)=-1.
       \end{equation*}

Let $s:\N\rightarrow\N$ be a nondecreasing function satisfying 
\[
\lim_{N\to\oo}s(N) \cdot \Delta \log(V(N))=\oo
\]
and $s(N)\leq N-1$ for all $N\in \N$.
Then $\E_{n\in [N-s(N),N]}x_n = L$.
\end{corollary}

\begin{proof} Our strategy is to find functions $U$ and $\tilde{U}$ where $U$ satisfies (\ref{eq:U_condition}), $\tilde{U}$ satisfies the assumptions of Theorem \ref{lem:weighted_implies_interval}, and
\[
\lim_{N\to\oo}\E^U_{n\leq N}x_n =\lim_{N\to\oo}\E^{\tilde{U}}_{n\leq N}x_n=L.
\]
Then by taking limits of equation (\ref{eq:weighted_implies_interval}) and invoking the Silverman-Toeplitz Theorem it will follow that $\lim_{N\to\oo}\E_{n\in [N-s(N),N]}x_n = L$. We begin by picking a function $r:\N\rightarrow (0,\oo)$ which decreases to $0$ slowly enough such that
\begin{enumerate}[label = (\roman*)]
\item $\lim_{N\to\oo}r(N)\cdot \log(V(N))=\oo$,
\item $\lim_{N\to\oo}r(N)\cdot s(N)\cdot \Delta\log V(N) =\oo$,
\item $\lim_{N\to\oo}\frac{\Delta \log(r(N))}{\Delta \log\log(V(N))}=\lim_{N\to\oo}\frac{\log(V(N))\cdot \Delta r(N)}{r(N)\cdot \Delta \log(V(N))}=0$,
from which it follows that $\lim_{N\to\oo}\frac{\Delta \log(r(N))}{\Delta \log(V(N)^{r(N)})}=0$ and $\lim_{N\to\oo}\frac{\log(V(N))\cdot \Delta r(N)}{\Delta \log (V(N))}=\lim_{N\to\oo}\frac{\Delta \log(V(N)^{r(N)})}{\Delta \log (V(N))}=0$.
\end{enumerate}
For example, we can take
\[
r(N) = \max\left(\log(V(N))^{-\frac{1}{2}},(s(N)\Delta \log(V(N)))^{-\frac{1}{2}},e^{\frac{\Delta \log(\log(V(N)))}{N}}\cdot r(N-1)^{-1}\right).
\]

Let $U(N) = \sum_{n\leq N}\frac{\Delta V(n)}{V(n)}V(n)^{r(n)}$, so that $\Delta U(N) = \frac{\Delta V(N)}{V(N)}V(N)^{r(N)}>0$.

We can verify that equation (\ref{eq:U_condition}) holds. Indeed, 
\begin{align*}
    &\lim_{N\to\oo}\frac{V(N-1)}{\Delta U(N)}\left(\frac{\Delta U(N)}{\Delta V(N)}-\frac{\Delta U(N-1)}{\Delta V(N-1)}\right) \\
    =& \lim_{N\to\oo}\frac{V(N-1)}{\frac{\Delta V(N)}{V(N)}V(N)^{r(N)}}\left(\frac{\frac{\Delta V(N)}{V(N)}V(N)^{r(N)}}{\Delta V(N)}-\frac{\frac{\Delta V(N-1)}{V(N-1)}V(N-1)^{r(N-1)}}{\Delta V(N-1)}\right)\\
    =& \lim_{N\to\oo}\frac{V(N-1)V(N)}{\Delta V(N)V(N)^{r(N)}}\left(\frac{V(N)^{r(N)}}{V(N)}-\frac{V(N-1)^{r(N-1)}}{ V(N-1)}\right)\\
    =& \lim_{N\to\oo}\frac{V(N-1)V(N)}{\Delta V(N)V(N)^{r(N)}}\left(\frac{\Delta (V(N)^{r(N)})}{V(N-1)}+\frac{-\Delta V(N)\cdot V(N)^{r(N)}}{V(N)V(N-1)}\right)\\
    =&\lim_{N\to\oo}\frac{\frac{\Delta (V(N)^{r(N)})}{V(N)^{r(N)}}}{\frac{\Delta V(N)}{V(N)}}-1= \lim_{N\to\oo}\frac{\Delta \log(V(N)^{r(N)})}{\Delta \log(V(N))}-1=-1.
\end{align*}
Additionally, it follows from Remark \ref{remark:hardy_functions} that $\lim_{N\to\oo}\frac{\Delta U(N)}{U(N)}=0$. To check that $U(N)$ tends to $\oo$ as $N\to\oo$, note that
\begin{align*}
    \Delta (V(N)^{r(N)}) = & V(N)^{r(N)}\cdot \left(\Delta \log(V(N)^{r(N)})\right)\\=&V(N)^{r(N)}\cdot \left(r(N)\cdot  \Delta \log(V(N))+\Delta r(N)\cdot  \log(V(N-1))\right)\\
   =&V(N)^{r(N)} \frac{\Delta V(N)}{V(N)}r(N) \cdot (1+o_{N\to\oo}(1))
\end{align*}
and so $\Delta U(N) = V(N)^{r(N)} \frac{\Delta V(N)}{V(N)}  = \frac{  \Delta (V(N)^{r(N)})}{r(N)}\cdot (1+o_{N\to\oo}(1)$. Similarly,
\begin{align*}
    \Delta \left(\frac{V(N)^{r(N)}}{r(N)}\right) =& \frac{  \Delta (V(N)^{r(N)})}{r(N)} - V(N)^{r(N)}\cdot \frac{ \Delta r(N) }{r(N)r(N-1)}\\
    =&\frac{  \Delta (V(N)^{r(N)})}{r(N)}\cdot (1- \frac{\frac{\Delta r(N)}{r(N-1)}}{\frac{\Delta(V(N)^{r(N)})}{V(N)^{r(N)}}})\\
    =&\frac{  \Delta (V(N)^{r(N)})}{r(N)}\cdot (1+o_{N\to\oo}(1))
\end{align*}
by our assumptions on $r$. Let $T(N) = \frac{V(N)^{r(N)}}{r(N)}$, which increases to $\oo$. We have shown that $\lim_{N\to\oo}\frac{\Delta T(N)}{\Delta U(N)}=1$ and hence $\lim_{N\to\oo}\frac{T(N)}{U(N)}=1$. It follows that $U(N)$ tends to $\oo$ as $N\to\oo$. Since $U$ increases to $\oo$ and satisfies (\ref{eq:U_condition}), we have by hypothesis that $\lim_{N\to\oo}\E_{n\leq N}^Ux_n =L$. Since $\lim_{N\to\oo}\frac{T(N)}{U(N)}=1$ and $\lim_{n\to\oo}\frac{\Delta T(n)}{\Delta U(n)}=1$, we also have that $\lim_{N\to\oo}\E_{n\leq N}^{T}x_n =L$.

Now we define a function $\tilde{U}$ which satisfies the conditions of Lemma \ref{lem:weighted_implies_interval}. Let $\tilde{U}(1) = U(1)$ and $\tilde{U}(N+1) = \alpha_N \tilde{U}(N)$, where 
\begin{equation}\label{eq:alpha_def}
\alpha_N = \frac{1}{1-\frac{\Delta T(N+1)}{T(N+1)}(1+Z(N))}
\end{equation}
and $Z$ is a function which decreases to $0$ sufficiently slowly. Then 
\begin{equation}\label{eq:tilde_U_ratio_goes_to_1}
\frac{\frac{\Delta \tilde{U}(N)}{\tilde{U}(N)}}{\frac{\Delta T(N)}{T(N)}}= \frac{1+Z(N)}{1-\frac{\Delta T(N+1)}{T(N+1)}(1+Z(N))} \to 1 \text{ as } N\to\oo.
\end{equation}
It follows that 
\begin{align*}
\lim_{N\to\oo}s(N)\cdot \Delta \log(\tilde{U}(N))=&\lim_{N\to\oo}s(N)\cdot \frac{\Delta \tilde{U}(N)}{\tilde{U}(N)}\\
=&\lim_{N\to\oo}s(N)\cdot \frac{\Delta T(N)}{T(N)}\\
=&\lim_{N\to\oo}s(N)\cdot r(N)\cdot \frac{\Delta V(N)}{V(N)}\\
=&\lim_{N\to\oo}s(N)\cdot r(N)\cdot \Delta \log(V(N))=\oo.
\end{align*}

Now to check that $\Delta\tilde{U}$ is nondecreasing,
\begin{align*}
    \Delta^2\tilde{U}(N+1)=&\Delta(\tilde{U}(N+1)-\tilde{U}(N)) = \Delta(\tilde{U}(N)(\alpha_{N}-1)) \\=&  \Delta \tilde{U}(N)(\alpha_{N}-1)+\tilde{U}(N-1)(\alpha_{N}-\alpha_{N-1})\\
    =& \alpha_N(\Delta \tilde{U}(N)+\tilde{U}(N-1))-\Delta\tilde{U}(N)-\alpha_{N-1}\tilde{U}(N-1)\\
     =& \alpha_N \tilde{U}(N)-\Delta\tilde{U}(N)-\alpha_{N-1}\tilde{U}(N-1).
\end{align*}
This is nonnegative so long as $\alpha_{N}\geq \frac{\Delta \tilde{U}(N)}{\tilde{U}(N)}+\alpha_{N-1}\frac{\tilde{U}(N-1)}{\tilde{U}(N)}$.
Using the fact that $\frac{\tilde{U}(N-1)}{\tilde{U}(N)} = \frac{1}{\alpha_{N-1}}$ and $\frac{\Delta \tilde{U}(N)}{\tilde{U}(N)}=1-\frac{1}{\alpha_{N-1}}$, this inequality becomes 
\begin{equation}\label{eq:alpha_eq}
\alpha_N\geq 2-\frac{1}{\alpha_{N-1}} = 2-(1-\frac{\Delta T(N)}{T(N)}(1+Z(N-1))) = 1+\frac{\Delta T(N)}{T(N)}(1+Z(N-1)).
\end{equation}
We can note that 
\[
\alpha_N =  1+\frac{\frac{\Delta T(N+1)}{T(N+1)}(1+Z(N))}{1-\frac{\Delta T(N+1)}{T(N+1)}(1+Z(N))}\geq 1+\frac{\Delta T(N+1)}{T(N+1)}(1+Z(N))
\]
and so it suffices to show that $\frac{\Delta T(N+1)}{T(N+1)}(1+Z(N))\geq \frac{\Delta T(N)}{T(N)}(1+Z(N-1))$. However, after rearranging this is 
\[
1+Z(N)\geq \frac{\frac{\Delta T(N)}{T(N)}}{\frac{\Delta T(N+1)}{T(N+1)}}(1+Z(N-1))=\frac{\Delta\log(V(N))}{\Delta \log(V(N+1))}(1+o_{N\to\oo}(1)).
\]
Since $\Delta \log(V(N))$ is decreasing, this inequality is satisfied so long as $Z$ decreases to $0$ slowly enough. Thus, $\Delta \tilde{U}$ is nondecreasing.

Lastly, to show that $\lim_{N\to\oo}\E_{n\leq N}^{\tilde{U}}x_n = L$ consider \cite[Lemma 7.1]{uniform_distribution}, which says that if $\lim_{N\to\oo}\E_{n\leq N}^{{T}}x_n = L$, $\frac{\Delta \tilde{U}(N+1)}{\Delta \tilde{U}(N)}\geq \frac{\Delta T(N+1)}{\Delta T(N)}$, and there is an $H>0$ such that $\frac{\Delta \tilde{U}(N)}{\tilde{U}(N)}\leq H\cdot\frac{\Delta T(N)}{T(N)}$ for all $N$ then $\lim_{N\to\oo}\E_{n\leq N}^{\tilde{U}}x_n = L$. Using (\ref{eq:tilde_U_ratio_goes_to_1}), we can find an $H>0$ such that $\frac{\Delta \tilde{U}(N)}{\tilde{U}(N)}\leq H\cdot\frac{\Delta T(N)}{T(N)}$ for all $N$. Also, the inequality $\frac{\Delta \tilde{U}(N+1)}{\Delta \tilde{U}(N)}\geq \frac{\Delta T(N+1)}{\Delta T(N)}$ holds since 
\[
\frac{\Delta \tilde{U}(N+1)}{\Delta \tilde{U}(N)} = \frac{(\alpha_N-1) \tilde{U}(N)}{(\alpha_{N-1}-1) \tilde{U}(N-1)} = \frac{(\alpha_N-1)\cdot (1+o_{N\to\oo}(1))}{(\alpha_{N-1}-1)}
\]
and 
\begin{align*}
\frac{\Delta T(N+1)}{\Delta T(N)} = \frac{\frac{\Delta T(N+1)}{T(N+1)}}{\frac{\Delta T(N)}{T(N)}}\cdot \frac{T(N+1)}{T(N)} = \frac{\Delta \log(T(N+1))}{\Delta \log(T(N))}\cdot (1+o_{N\to\oo}(1))\leq 1+o_{N\to\oo}(1).
\end{align*}
So it suffices to show that $\frac{\alpha_N-1}{\alpha_{N-1}-1}\geq 1+o_{N\to\oo}(1)$, but we can rearrange (\ref{eq:alpha_eq}) to see that $\frac{\alpha_N-1}{\alpha_{N-1}-1}\geq \frac{1}{\alpha_{N-1}}=1+o_{N\to\oo}(1)$. This completes the proof.
\end{proof}

\section[Section 4]{$(4)\implies(5)$}\label{section:4_implies_5}
In this section we prove the implication $(4)\implies (5)$ in Theorem \ref{thm:general_main}. As in the previous section, the proof of this implication relies heavily on the Silverman-Toeplitz Theorem.
\begin{theorem}
   Let $Y$ be a Banach space, let $(x_n)_{n\in \N}\subset Y$ be a bounded sequence, and let $L\in Y$. Let $V:\N\rightarrow (0,\oo)$ be a function which increases to $\oo$ with $\Delta \log(V(N))$ decreasing to $0$ such that $\lim_{N\to\oo}\frac{\Delta \log(V(N))}{\Delta \log(N)}=\oo$. Suppose that $\lim_{N\to\oo} \E_{n\in [N-s(N),N]}x_n = L$ for each nondecreasing function $s:\N\rightarrow\N$ satisfying
       \begin{equation} 
       \lim_{N\to\oo}s(N) \cdot \Delta \log(V(N))=\oo
       \end{equation}
       and $s(N)\leq N-1$ for all $N\in \N$.
   Then $\E_{\text{unif}}^{\log V}x_n = L$.
\end{theorem}

\begin{proof}

Let $W(N) = \log(V(N))$ for $N\in \N$. Let $A = A(N)$ and $B = B(N)$ be arbitrary integer valued functions with $W(B)-W(A)\to \oo$ as $N\to\oo$. We will show that $\lim_{N\to\oo}\E_{n\in [A,B]}^Wx_n = L$. Since $A$ and $B$ are arbitrary, this is sufficient to show that $\E_{\text{unif}}^{\log V}x_n =L$.

    Pick a nondecreasing function $s\colon \N\rightarrow \N$ satisfying $\lim_{N\to\oo}s(N)\cdot \Delta W(N) = \oo$, but which is slow enough such that $s(N)\leq N-1$ for all $N\in \N$ and $\lim_{N\to\oo}\frac{s(B)\cdot \Delta W(B)}{W(B)-W(A)} = 0$. Further, we can assume that $\Delta s(N)\in \{0,1\}$ for all $N\in \N$, since the condition $\lim_{N\to\oo}\frac{\Delta \log(V(N))}{\Delta \log (N)}=\lim_{N\to\oo}N\cdot \Delta \log(V(N))=\oo$ implies that we can take $s(N)$ to be asymptotically slower than $N$. By assumption, we have that $\lim_{N\to\oo}\E_{n\in [N-s(N),N]}x_n=L$.

    We will show that there exists a doubly indexed sequence of nonnegative constants $(c_{N,n})_{N,n\in \N}$ such that $\lim_{N\to\oo}\sum_{n\in \N}c_{N,n}=1$, $\lim_{N\to\oo}c_{N,n_0}=0$ for each fixed $n_0\in \N$, and 
    \begin{equation}\label{eq:(4)_implies_(5)}
    \sum_{k=A}^{B}c_{N,k}\E_{n\in[k-s(k),k]}x_n=\E_{n\in [A,B]}^Wx_n+o_{N\to \oo}(1).
    \end{equation}

    Once we have shown these properties, we can conclude that $\lim_{N\to\oo}\E_{n\in [A,B]}^Wx_n = L$ by the Silverman-Toeplitz Theorem and our assumption that $\lim_{N\to\oo}\E_{n\in [N-s(N),N]}x_n=L$.

    We will define $(c_{N,n})$ inductively. For $N\in \N$, define $c_{N,n} = 0$ if $n>B$ or $n<A$. Then put $c_{N,B} = \frac{s(B)\Delta W(B)}{W(B)-W(A)}$. For $n\in [A,B]$, having already defined $c_{N,n+1},\dots, c_{N,B}$, let 
    \begin{equation}
    c_{N,n} = s(n)\cdot \left(\frac{\Delta W(n)}{W(B)-W(A)}-\sum_{j=n+1}^{B}\frac{c_{N,j}}{s(j)}\cdot 1_{\{n\geq j-s(j)\}}\right).
    \end{equation}

    With this definition, it is clear that for each $N\in \N$ we have
    \begin{equation}\label{eq:(4)_implies_(5)_inductive}
    \sum_{j=n}^{B}\frac{c_{N,j}}{s(j)}\cdot 1_{\{n\geq j-s(j)\}}=\frac{\Delta W(n)}{W(B)-W(A)} \text{ for all } n\in [A,B].
    \end{equation}

    Note that $\lim_{N\to\oo}c_{N,n_0}=0$ for each fixed $n_0\in \N$, since $\lim_{N\to\oo}A = \oo$ and $c_{N,n_0}=0$ for $A>n_0$. Next, we will show that $c_{N,n}\geq 0$ for all $n$. We proceed with an inductive argument. First note that $c_{N,B} = \frac{s(B)\cdot \Delta W(B)}{W(B)-W(A)}\geq 0$ and that $c_{N,n}=0$ for $n>B$ or $n<A$. Now suppose that $A\leq k<B$ and $c_{N,k+1},\dots, c_{N,B}\geq 0$. We know that $\Delta W(k)\geq \Delta W(k+1)$ by our assumption that $\Delta \log(V(N))$ is decreasing and so 
\begin{equation}\label{eq:(4)_implies_(5)_1}
\sum_{j=k}^{B}\frac{c_{N,j}}{s(j)}\cdot 1_{\{k\geq j-s(j)\}}=\frac{\Delta W(k)}{W(B)-W(A)}\geq \frac{\Delta W(k+1)}{W(B)-W(A)} = \sum_{j=k+1}^{B}\frac{c_{N,j}}{s(j)}\cdot 1_{\{k+1\geq j-s(j)\}}.
\end{equation}
Let $z\leq B$ be the largest integer such that $k\geq z-s(z)$, so that 
\[
\sum_{j=k}^{B}\frac{c_{N,j}}{s(j)}\cdot 1_{\{k\geq j-s(j)\}}-\sum_{j=k+1}^{B}\frac{c_{N,j}}{s(j)}\cdot 1_{\{k+1\geq j-s(j)\}}
\]
is equal to either $\frac{c_{N,k}}{s(k)}-\frac{c_{N,z+1}}{s(z+1)}$ or $\frac{c_{N,k}}{s(k)}-\frac{c_{N,z+1}}{s(z+1)}-\frac{c_{N,z+2}}{s(z+2)}$. In either case, we have that 
\[
0\leq \sum_{j=k}^{B}\frac{c_{N,j}}{s(j)}\cdot 1_{\{k\geq j-s(j)\}}-\sum_{j=k+1}^{B}\frac{c_{N,j}}{s(j)}\cdot 1_{\{k+1\geq j-s(j)\}} \leq \frac{c_{N,k}}{s(k)}-\frac{c_{N,z+1}}{s(z+1)}.
\]
By assumption, $c_{N,z+1}\geq 0$ and so $c_{N,k}\geq 0$, which completes the induction.

To prove that equation (\ref{eq:(4)_implies_(5)}) holds, observe that
\begin{align*}
    &\sum_{k=A}^{B}c_{N,k}\E_{n\in[k-s(k),k]}x_n=\sum_{k=A}^{B}c_{N,k}\frac{1}{s(k)}\sum_{n=k-s(k)}^kx_n\\ =& \sum_{k=A}^{B}c_{N,k}\frac{1}{s(k)}\sum_{n=k-s(k)}^kx_n \cdot 1_{\{k\geq n\geq k-s(k)\}}.
    \end{align*}
    Switching the order of summation, we have
    \begin{align*}
    &\sum_{k=A}^{B}c_{N,k}\frac{1}{s(k)}\sum_{n=k-s(k)}^kx_n \cdot 1_{\{k\geq n\geq k-s(k)\}}=\sum_{n=A-s(A)}^Bx_n\sum_{k\geq n}\frac{c_{N,k}}{s(k)}\cdot 1_{\{n\geq k-s(k)\}}.
    \end{align*}
    Now we can split the outer sum into two sums in order to apply equation (\ref{eq:(4)_implies_(5)_inductive}),
    \begin{align*}
    &\sum_{n=A-s(A)}^Bx_n\sum_{k\geq n}\frac{c_{N,k}}{s(k)}\cdot 1_{\{n\geq k-s(k)\}}\\
    =&\sum_{n=A}^Bx_n\sum_{k= n}^B\frac{c_{N,k}}{s(k)}\cdot 1_{\{n\geq k-s(k)\}}+\sum_{n=A-s(A)}^{A-1}x_n\sum_{k= n}^B\frac{c_{N,k}}{s(k)}\cdot 1_{\{n\geq k-s(k)\}}\\
    =& \sum_{n=A}^Bx_n\frac{\Delta W(n)}{W(B)-W(A)}+\sum_{n=A-s(A)}^{A-1}x_n\sum_{k= n}^B\frac{c_{N,k}}{s(k)}\cdot 1_{\{n\geq k-s(k)\}}\\
    =& \E_{n\in [A,B]}^Wx_n+\sum_{n=A-s(A)}^{A-1}x_n\sum_{k= n}^B\frac{c_{N,k}}{s(k)}\cdot 1_{\{n\geq k-s(k)\}}.
\end{align*}

Next, we will show that $\sum_{n=A-s(A)}^{A-1}x_n\sum_{k= n}^B\frac{c_{N,k}}{s(k)}\cdot 1_{\{n\geq k-s(k)\}}=o_{N\to\oo}(1)$, which will prove equation (\ref{eq:(4)_implies_(5)}). First, we will use the triangle inequality,
\[
\left|\sum_{n=A-s(A)}^{A-1}x_n\sum_{k= n}^B\frac{c_{N,k}}{s(k)}\cdot 1_{\{n\geq k-s(k)\}}\right| \leq \sum_{n=A-s(A)}^{A-1}\|x_n\|\sum_{k= n}^B\frac{c_{N,k}}{s(k)}\cdot 1_{\{n\geq k-s(k)\}}.
\]

Recall that $c_{N,k}=0$ for $k<A$ and so for $n<A$ we have
\[
\sum_{k= n}^B\frac{c_{N,k}}{s(k)}\cdot 1_{\{n\geq k-s(k)\}} = \sum_{k= A}^B\frac{c_{N,k}}{s(k)}\cdot 1_{\{n\geq k-s(k)\}}\leq \sum_{k= A}^B\frac{c_{N,k}}{s(k)}\cdot 1_{\{A\geq k-s(k)\}}.
\]
Then 
\[
\sum_{n=A-s(A)}^{A-1}\|x_n\|\sum_{k= n}^B\frac{c_{N,k}}{s(k)}\cdot 1_{\{n\geq k-s(k)\}} \leq \sup_{n\in \N}\|x_n\|\cdot s(A)\cdot \sum_{k= A}^B\frac{c_{N,k}}{s(k)}\cdot 1_{\{A\geq k-s(k)\}}
\]

By (\ref{eq:(4)_implies_(5)_inductive}), 
\[
\sum_{k= A}^B\frac{c_{N,k}}{s(k)}\cdot 1_{\{A\geq k-s(k)\}} = \frac{\Delta W(A)}{W(B)-W(A)}\leq\frac{\Delta W(B)}{W(B)-W(A)} .
\]
So altogether, we have shown that 
\begin{align*}
\left|\sum_{n=A-s(A)}^{A-1}x_n\sum_{k= n}^B\frac{c_{N,k}}{s(k)}\cdot 1_{\{n\geq k-s(k)\}}\right|\leq 
\sup_{n\in \N}\|x_n\|\cdot \frac{s(B)\Delta W(B)}{W(B)-W(A)} = o_{N\to\oo}(1).
\end{align*}

Finally, taking $x_n = 1$ in equation (\ref{eq:(4)_implies_(5)}) shows that $\sum_{n\in \N}c_{N,n}= 1+o_{N\to\oo}(1)$, which concludes the proof.
\end{proof}

\section[Section 5]{$(5)\implies(2)$}\label{section:5_implies_2}
In this section we prove the implication $(5)\implies (2)$ in Theorem \ref{thm:general_main}. Our strategy is to emulate the proof of a classical theorem of Schatte.
\begin{theorem}[{\cite[Theorem B]{Schatte90}}]\label{thm:schatte}
    Let $(x_n)_{n\in \N}$ be a bounded sequence of complex numbers and let $L\in \C$. Then $\E(\oo)(x_n) = L$ if and only if $\E_{\unif}^{\log}(x_n) = L$.
\end{theorem}
\begin{example}\label{ex:log_not_ud}
    It is known that $\lim_{N\to\oo}\E_{n\leq N}e^{2\pi i \log(n)}=\frac{1}{N}\sum_{n=1}^Ne^{2\pi i \log(n)}$ does not tend toward $0$ as $N\to\oo$. Indeed, 
    \[
    \frac{1}{N}\sum_{n=1}^Ne^{2\pi i \log(n)} = \frac{1}{N}\sum_{n=1}^Ne^{2\pi i \log(N)}e^{2\pi i \log(n/N)} = e^{2\pi i \log(N)}\frac{1}{N}\sum_{n=1}^Ne^{2\pi i \log(n/N)}.
    \]
    The sum $\frac{1}{N}\sum_{n=1}^Ne^{2\pi i \log(n/N)}$ is a Riemann sum for the integral $\int_0^1e^{2\pi i \log(x)}dx$ and so $\frac{1}{N}\sum_{n=1}^Ne^{2\pi i \log(n/N)} = \int_0^1e^{2\pi i \log(x)}dx + o_{N\to\oo}(1)$.

    Let $x_n = e^{2\pi i \log(n)}$ and let $C = \int_0^1e^{2\pi i \log(x)}dx = \frac{1}{1+2\pi i }$. Note that $\|x_n\| = 1$ for all $n\in \N$ and $|C|<1$. Then
    \begin{align*}
        &\E_{n\leq N}x_n = \frac{1}{N}\sum_{n=1}^Ne^{2\pi i \log(n)}=e^{2\pi i \log(N)}\frac{1}{N}\sum_{n=1}^Ne^{2\pi i \log(n/N)}=C\cdot x_N+o_{N\to\oo}(1)\\
        &\E_{n\leq N}(2)x_n = \E_{n\leq N}(C\cdot x_n+o_{n\to\oo}(1)) = C^2\cdot x_N+o_{N\to\oo}(1)\\
        &\E_{n\leq N}(3)x_n = \E_{n\leq N}(C^2\cdot x_n+o_{n\to\oo}(1)) = C^3\cdot x_N+o_{N\to\oo}(1)\\
        &\vdots
    \end{align*}
    In general, $\E_{n\leq N}(k)x_n = C^k\cdot x_N+o_{N\to\oo}(1)$. From this we can see that as $k\to\oo$, $\E_{n\leq N}(k)x_n\to 0$. So $\E(\oo)x_n = 0$. According to equation Theorem \ref{thm:schatte}, this implies that $$ 
    \E_{\text{unif}}^{\log V}x_n =\lim_{\log(N)-\log(M)\to\oo}\frac{1}{\log(N)-\log(M)}\sum_{n=M}^N\frac{e^{2\pi i \log(n)}}{n}= 0.
    $$
\end{example}

For our purposes, we need to begin with the following theorem, which is a straightforward extension of \cite[Satz 1]{Schatte74}.
\begin{theorem}\label{thm:schatte_iterating}
     Let $Y$ be a Banach space, let $(x_n)_{n\in \N}\subset Y$ be a bounded sequence. Let $V:\N\rightarrow (0,\oo)$ be a function which increases to $\oo$ and satisfies $\lim_{N\to\oo}\frac{\Delta V(N)}{V(N)}=0$. Then for each $k\in \N$, 
    \begin{equation}\label{eq:schatte_iterating}
       \E(k+1)_{n\leq N}^Vx_n = \frac{1}{V(N)\cdot k!}\sum_{n=1}^N \Delta V(n)\log\left(\frac{V(N)}{V(n)}\right)^k\cdot x_n+o_{N\to\oo}(1).
    \end{equation}
\end{theorem}
\begin{proof}

We proceed by induction on $k$. There is nothing to prove when $k=0$, so suppose that equation (\ref{eq:schatte_iterating}) holds and we will show that
\begin{equation}\label{eq:schatte_iterating_induction}
 \E(k+2)_{n\leq N}^Vx_n = \frac{1}{V(N)\cdot (k+1)!}\sum_{n=1}^N \Delta V(n)\log\left(\frac{V(N)}{V(n)}\right)^{k+1}\cdot x_n+o_{N\to\oo}(1).
\end{equation}
Then
 \begin{align*}
      &\E(k+2)_{n\leq N}^Vx_n = \E(k+1)_{n\leq N}^V(\E_{m\leq n}^Vx_m) \\
      =& \frac{1}{V(N)\cdot k!}\sum_{n=1}^N\Delta V(n)\log\left(\frac{V(N)}{v(n)}\right)^k\E_{m\leq n}^Vx_m+o_{N\to\oo}(1)
      \\
      =&\frac{1}{V(N)\cdot k!}\sum_{n=1}^N\Delta V(n)\log\left(\frac{V(N)}{v(n)}\right)^k\left(\frac{1}{V(n)}\sum_{m=1}^n\Delta V(m)x_m\right)+o_{N\to\oo}(1)
      \\
      =&\frac{1}{V(N)\cdot k!}\sum_{m=1}^N\Delta V(m)x_m\cdot\left(\sum_{n=m}^N\frac{\Delta V(n)}{V(n)}\log\left(\frac{V(N)}{V(n)}\right)^k\right)+o_{N\to\oo}(1).
 \end{align*}   

 We will show that $\sum_{n=m}^N\frac{\Delta V(n)}{V(n)}\log\left(\frac{V(N)}{V(n)}\right)^k = \frac{\log\left(\frac{V(N)}{V(m)}\right)^{k+1}}{k+1}\cdot (1+o_{m\to\oo}(1))$, from which equation (\ref{eq:schatte_iterating_induction}) follows. Recall from Remark \ref{remark:hardy_functions} that $\frac{\Delta V(N)}{V(N)} = \Delta \log(V(N))\cdot (1+o_{N\to\oo}(1))$. Then 
 \begin{equation*}
 \sum_{n=m}^N\frac{\Delta V(n)}{V(n)}\log\left(\frac{V(N)}{V(n)}\right)^k  = (1+{o_{m\to\oo}(1))\cdot }\sum_{n=m}^N\Delta \log(V(n))\log\left(\frac{V(N)}{V(n)}\right)^k.
 \end{equation*}
Let $\ell_n = \log\left(\frac{V(N)}{V(n)}\right)$ and $G_n = \sum_{i=m}^n\Delta \log(V(i))$, so that $\Delta G_n = \Delta \log (V(n)) = -\Delta \ell_n$. Then using summation by parts, we have 
\begin{equation}
\label{eq:schatte_weights_summation_by_parts}
     \sum_{n=m}^N\Delta G_n\ell_n^k = G_n\ell_N-G_{m-1}\ell_{m-1}
     -\sum_{n=m}^{N-1}G_n\left(\ell_{n+1}^k-\ell_{n}^k\right).
 \end{equation}
 The first two terms are equal to $0$ since $\ell_N = 0$ and $G_{m-1}=0$. Next, noting that $\lim_{n\to\oo}\frac{\ell_{n}}{\ell_{n+1}} = 1$, we have 
\begin{align*}
\ell_{n+1}^k-\ell_n^k= (\ell_{n+1}-\ell_{n})\sum_{i=0}^{k-1}\ell_{n+1}^i\ell_{n}^{k-1-i} =& (\ell_{n+1}-\ell_{n})\ell_{n+1}^{k-1} k (1+o_{n\to\oo}(1)) \\=& -\Delta G_{n+1}\cdot \ell_{n+1}^{k-1} k (1+o_{n\to\oo}(1)).
\end{align*}
Additionally, we can write 
\[
G_n = \sum_{i=m}^n\Delta \log(V(i)) = \log(V(n))-\log(V(m-1)) = \ell_{m-1}-\ell_n = (\ell_{m}-\ell_{n+1})\cdot (1+o_{m\to\oo}).
\]
So,
\begin{align*}
   (1+o_{m\to\oo}(1))\cdot&  \sum_{n=m}^{N-1}G_n\left(\ell_{n+1}^k-\ell_n^k\right) = \sum_{n=m}^{N-1}(\ell_m-\ell_{n+1})(-\Delta G_{n+1}\cdot \ell_{n+1}^{k-1} k)\\
=&-k\ell_{m}\sum_{n=m}^{N-1}\Delta G_{n+1}\cdot \ell_{n+1}^{k-1} +k\sum_{n=m}^{N-1}\Delta G_{n+1}\cdot \ell_{n+1}^{k}\\
=&\left(-k\ell_{m}\sum_{n=m}^{N}\Delta G_{n}\cdot \ell_{n}^{k-1} +k\sum_{n=m}^{N}\Delta G_{n}\cdot \ell_{n}^{k}\right)(1+o_{m\to\oo}(1)).
\end{align*}
Comparing with (\ref{eq:schatte_weights_summation_by_parts}), we have
\begin{align*}
      \sum_{n=m}^N\Delta G_n\ell_n^k = -\left(-k\ell_{m}\sum_{n=m}^{N}\Delta G_{n} \ell_{n}^{k-1} +k\sum_{n=m}^{N}\Delta G_{n} \ell_{n}^{k}\right)\cdot (1+o_{m\to\oo}(1))
\end{align*}
so 
\begin{align*}
     (k+1)\sum_{n=m}^N\Delta G_n\ell_n^k =&(1+o_{m\to\oo}(1))\cdot k\ell_{m}\sum_{n=m}^{N}\Delta G_{n} \ell_{n}^{k-1}\\
     =&(1+o_{m\to\oo}(1))\cdot k\ell_{m}\sum_{n=m}^{N}-\Delta \ell_{n} \ell_{n}^{k-1}\\
     =&(1+o_{m\to\oo}(1))\cdot k\ell_{m}\left(\frac{\ell_m^k}{k}-\frac{\ell_N^k}{k}\right)\\
     =&(1+o_{m\to\oo}(1))\cdot \ell_m^{k+1} = (1+o_{m\to\oo}(1))\cdot \log\left(\frac{V(N)}{V(m)}\right)^{k+1}
\end{align*}
as desired.
\end{proof}

\begin{lemma}\label{lem:schatte_iterated_is_zero}
  Let $Y$ be a Banach space and let $(x_n)_{n\in \N}\subset Y$ be a bounded sequence. Let $V:\N\rightarrow (0,\oo)$ be a function which increases to $\oo$ and satisfies $\lim_{N\to\oo}\frac{\Delta V(N)}{V(N)}=0$. Suppose that
    \begin{equation} \label{eq:is_zero_condition}
\limsup_{N\to\oo}\left|\sum_{n=1}^N\frac{\Delta V(n)}{V(n)}x_n\right|<\oo.
    \end{equation}
    Then $\E(\oo)^Vx_n = 0$.
\end{lemma}
\begin{proof}
Pick $k\in \N$ arbitrarily large and define $F(x) = \frac{x}{k!}\log(\frac{1}{x})^k$. Note that $F$ is increasing on the interval $(0,e^{-k})$ and decreasing on the interval $(e^{-k},1)$. From Theorem \ref{thm:schatte_iterating}, we have that 
\begin{equation}
    \E(k+1)_{n\leq N}^Vx_n = \sum_{n=1}^N\frac{\Delta V(n)}{V(n)}x_n\cdot F\left(\frac{V(n)}{V(N)}\right)+o_{N\to\oo}(1).
\end{equation}

Summation by parts gives us that 
\begin{align*}
&\sum_{n=1}^N\frac{\Delta V(n)}{V(n)}x_n\cdot F\left(\frac{V(n)}{V(N)}\right) \\
=& F\left(\frac{V(N)}{V(N)}\right)\left(\sum_{n=1}^N\frac{\Delta V(n)}{V(n)}x_n\right)-F\left(\frac{V(1)}{V(N)}\right)\left(\sum_{n=1}^1\frac{\Delta V(n)}{V(n)}x_n\right)\\-&\sum_{n=1}^N\left(\sum_{k=1}^{n+1}\frac{\Delta V(k)}{V(k)}\right)\cdot\left(F\left(\frac{V(k)}{V(N)}\right)-F\left(\frac{V(k-1)}{V(N)}\right)\right).
\end{align*}

Pick $C>0$ such that $\limsup_{N\to\oo}\left|\sum_{n=1}^N\frac{\Delta V(n)}{V(n)}x_n\right|\leq C$. Then
\begin{align*}
&\left|\sum_{n=1}^N\frac{\Delta V(n)}{V(n)}x_n\cdot F\left(\frac{V(n)}{V(N)}\right)\right|\\
\leq&\,
C\left(F(1)+F\left(\frac{V(1)}{V(N)}\right)+\sum_{n=1}^NF\left(\frac{V(n)}{V(N)}\right)-F\left(\frac{V(n-1)}{V(N)}\right)\right)\\
\leq &\,C\left(F(1)+F\left(\frac{V(1)}{V(N)}\right)+F(e^{-k})-F(1)-F\left(\frac{V(1)}{V(N)}\right)\right) \\=& \,C\cdot F(e^{-k}).
\end{align*}
The result follows from the fact that $\lim_{k\to\oo}F(e^{-k})=0$.
\end{proof}


\begin{theorem}\label{thn:uniform_implies_iterated}
    Let $Y$ be a Banach space and let $(x_n)_{n\in \N}\subset Y$ be a bounded sequence. Let $V:\N\rightarrow (0,\oo)$ be a function which increases to $\oo$ and satisfies $\lim_{N\to\oo}\frac{\Delta V(N)}{V(N)}=0$. Suppose that $\E_{\text{unif}}^{\log V}(x_n) = L$. Then $\E(\oo)^Vx_n = L$.
\end{theorem}

\begin{proof}
Without loss of generality, assume $L=0$. Pick any $\epsilon>0$. We will find bounded sequences $(y_n)_{n\in \N}, (z_n)_{n\in \N}$ with $x_n=y_n+z_n$ for each $n\in \N$, such that $\sup_{N\in \N}|y_N|<\epsilon$ and $\limsup_{N\to\oo}\left|\sum_{n=1}^N\frac{\Delta V(n)}{V(n)}z_n\right|<\oo$. From Lemma \ref{lem:schatte_iterated_is_zero} it will follow that 
$$
|\E(\oo)^Vx_n|\leq|\E(\oo)^Vy_n|+|\E(\oo)^Vz_n|< \epsilon+0
$$
which will conclude the proof.

Recall from Remark \ref{remark:hardy_functions} that $\frac{\Delta V(n)}{V(n)} = \Delta \log(V(n))\cdot (1+o_{n\to\oo}(1))$.
We know that 
\[
\frac{1}{\log(V(B))-\log(V(A))}\sum_{n=B}^A\Delta \log(V(n))x_n
\]
tends to $0$ whenever $\log(V(B))-\log(V(A))$ tends to $\oo$, and so it follows that we also have that 
\[
\frac{1}{\sum_{n=B}^A\frac{\Delta V(n)}{V(n)}}\sum_{n=B}^A\frac{\Delta V(n)}{V(n)}x_n
\]
tends to $L$ whenever $\log(V(B))-\log(V(A))$ tends to $\oo$.

We can find a $K_0>0$ such that for any $M,N$ with $W(N)-W(M)>K_0$,
\begin{equation}\label{eq:epsilonics_1}
     \left|\frac{1}{\sum_{n=M}^{N}\frac{\Delta V(n)}{V(n)}}\cdot  \sum_{n=M}^N\frac{\Delta V(n)}{V(n)}x_n\right|<\epsilon.
\end{equation}
Define a sequence $(N_i)$ as follows. Let $N_1=1$ and having picked $N_i$ for $i\in \N$, pick the smallest $N_{i+1}>N_i$ such that $\log(V(N_{i+1}-1))-\log(V(N_i))>K_0$. Define 
\begin{equation}\label{eq:y_def}
y_k = \frac{1}{\sum_{n=N_{i}}^{N_{i+1}-1}\frac{\Delta V(n)}{V(n)}} \cdot \sum_{n=N_{i}}^{N_{i+1}-1}\frac{\Delta V(n)}{V(n)}x_n \text{ for all }k\in [N_i,N_{i+1}-1].
\end{equation}
Then $|y_n|<\epsilon$ for all $n\in \N$. Put $z_n = x_n-y_n$ for all $n\in \N$, and note that $(z_n)$ is bounded since $(x_n)_{n\in \N}$ and $(y_n)_{n\in \N}$ are bounded. All that is left is to show that $\limsup_{N\to\oo}\left|\sum_{n=1}^N\frac{\Delta V(n)}{V(n)}z_n\right|<\oo$.

Pick $N\in \N$ and pick $i$ such that $N_{i-1}<N\leq N_{i}$. Then

\begin{align*}
 \sum_{n=1}^N\frac{\Delta V(n)}{V(n)}z_n = \sum_{m=1}^{i}\sum_{k=N_{m-1}}^{N_m-1}\frac{\Delta V(k)}{V(k)}z_k-\sum_{k=N+1}^{N_i-1}\frac{\Delta V(k)}{V(k)}z_k.
\end{align*}
$z_k = x_k-y_k$, and $y_k$ is constant for $k\in [N_{m-1},N_m-1]$, so  
\begin{align*}
    \sum_{m=1}^{i}\sum_{k=N_{m-1}}^{N_m-1}\frac{\Delta V(k)}{V(k)}z_k =& \sum_{m=1}^{i}\sum_{k=N_{m-1}}^{N_m-1}\frac{\Delta V(k)}{V(k)}x_k-\sum_{m=1}^{i}\sum_{k=N_{m-1}}^{N_m-1}\frac{\Delta V(k)}{V(k)}y_k\\
    =&\sum_{m=1}^{i}y_{N_{m-1}}\cdot\sum_{k=N_{m-1}}^{N_m-1}\frac{\Delta V(k)}{V(k)}-\sum_{m=1}^{i}y_{N_{m-1}}\cdot\sum_{k=N_{m-1}}^{N_m-1}\frac{\Delta V(k)}{V(k)} = 0.
\end{align*}
Then 
\begin{align*}
     \left|\sum_{n=1}^N\frac{\Delta V(n)}{V(n)}z_n \right|= \left|\sum_{k=N+1}^{N_i-1}
     \frac{\Delta V(k)}{V(k)}z_k\right|\leq \sup_{n\in \N}\|z_n\| \cdot \sum_{k=N+1}^{N_i-1}\frac{\Delta V(k)}{V(k)}.
\end{align*}
 $\sum_{k=N+1}^{N_i}\frac{\Delta V(k)}{V(k)}\leq \sum_{k=N_{i-1}}^{N_i}\frac{\Delta V(k)}{V(k)}$ which is close to $\log(V(N_{i}))-\log(V(N_{i-1}+1))$. By definition of $(N_i)_{i\in \N}$, we have $\log(V(N_{i}))-\log(V(N_{i-1}+1))\leq K_0+1$ for all large enough $i$. So $\limsup_{N\to\oo}\left|\sum_{n=1}^N\frac{\Delta V(n)}{V(n)}z_n\right|<\oo$ as desired.

\end{proof}


\bibliographystyle{aomalpha}
\bibliography{Paper_references}



\bigskip
\footnotesize
\noindent
Michael Reilly\\
\textsc{The Ohio State University}\\
\href{mailto:reilly.201@osu.edu}
{\texttt{reilly.201@osu.edu}}

\end{document}